\documentclass[10pt,twoside]{siamltex1213}
\usepackage{amsfonts,epsfig}
\usepackage{amsmath,amssymb,amsfonts,amsbsy,enumerate,mathrsfs,color}
\usepackage{hyperref}
\usepackage{url}
\usepackage{graphicx}

\newtheorem{prop}[theorem]{Proposition}
\newtheorem{problem}[theorem]{Problem}

\newcommand{\G}{\Gamma}

\def\tr{\mathop{\rm tr }\nolimits}

\begin{document}

\bibliographystyle{plain}
\title{
Graphs with many valencies and few eigenvalues\thanks{This version is published in Electronic Journal of Linear Algebra 28 (2015), 12--24.}}

\author{
Edwin R.\ van Dam\thanks{Department of Econometrics and O.R., Tilburg University,
	P.O. Box 90153, 5000 LE Tilburg, The Netherlands (Edwin.vanDam@uvt.nl)}
\and
Jack H.\ Koolen\thanks{Wen-tsun Wu Key Laboratory of CAS, School of Mathematical Sciences, University of Science and Technology of China, Hefei, 230026, Anhui, China
(koolen@ustc.edu.cn).}
\and
Zheng-jiang Xia\thanks{School of Mathematical Sciences,
University of Science and Technology of China, Hefei, 230026, Anhui, China
(xzj@mail.ustc.edu.cn).}}

\pagestyle{myheadings}
\markboth{E.R.\ van Dam, J.H.\ Koolen, and Z.\ Xia}{Graphs with many valencies and few eigenvalues}
\maketitle

\begin{center}
In memory of David Gregory
\end{center}

\begin{abstract}
Dom de Caen posed the question whether connected graphs with
three distinct eigenvalues have at most three distinct valencies. We do not answer this question, but instead construct connected graphs with four and five distinct eigenvalues and arbitrarily many distinct valencies. The graphs with four distinct eigenvalues come from regular two-graphs. As a side result, we characterize the disconnected graphs and the graphs with three distinct eigenvalues in the switching class of a regular two-graph.
\end{abstract}


\section{Introduction}

Dom de Caen (see \cite[Problem 9]{Dom} and \cite{dCvDS}) posed the question whether connected graphs with
three distinct eigenvalues have at most three distinct valencies. By eigenvalues of a graph we mean here the eigenvalues of the adjacency matrix. More generally, one may wonder whether
the number of distinct valencies in a graph is bounded by a function of the number of distinct eigenvalues. It is important to
restrict to connected graphs because, for example, the disjoint union of the complete bipartite graphs
on $2^i+2^{t-i}$ ($i=0,1,\dots,t$) vertices has three (distinct) eigenvalues and $t+1$ (distinct) valencies. Hence the number of valencies cannot be bounded as a function of the number of eigenvalues. Instead of the adjacency eigenvalues, one may of course pose similar questions for other eigenvalues of graphs. It is for example known that connected graphs with three Laplacian eigenvalues have at most two valencies; see \cite{mumubar}.

Mohar [private communication] observed that by adding one vertex and joining it in an arbitrary way to each
of the components of the above graph, one obtains a connected graph with many valencies, but interlacing of
eigenvalues (see \cite{HaeInterlacing}) implies that the number of eigenvalues is at most seven. Thus, it
follows that with only seven eigenvalues, the number of valencies can be arbitrarily large.

In this paper, we will further exploit Mohar's idea to construct connected bipartite graphs with five eigenvalues and arbitrarily many valencies. In the same spirit, we will construct connected non-bipartite graphs with five eigenvalues and many valencies. Moreover, we will use regular two-graphs and Seidel switching to construct connected graphs with four eigenvalues and many valencies. In particular, we shall use switching in the symplectic strongly regular graphs with respect to subgraphs that have many valencies. Still, De Caen's question remains open, as is the question whether there are connected bipartite graphs with four eigenvalues and more than four valencies.

As a side result, we will show that if a graph in the switching class of a non-trivial regular two-graph is disconnected, then it must be the disjoint union of a strongly regular graph and an isolated vertex. Also, if a graph in the switching class of a non-trivial regular two-graph has only three eigenvalues, then it must be regular, and hence strongly regular.

For background on eigenvalues of graphs we refer to the monographs by Brouwer
and Haemers \cite{bh12} and Cvetkovi\'c, Doob, and Sachs \cite{cds82}.

\section{Bipartite graphs with five eigenvalues}

Let us first exploit Mohar's construction further, in the sense that we want to bring the number of
eigenvalues down from seven to five.

\begin{theorem} For every integer $t$, there is a connected bipartite graph with five distinct eigenvalues and at least $t$ distinct valencies.
\end{theorem}

\begin{proof}
Consider the disjoint union of $f$ mutually non-isomorphic complete bipartite graphs with $e$ edges, with $f>1$. Then this graph has spectrum $$\{\sqrt{e}^{(f)},0^{(g)},-\sqrt{e}^{(f)}\},$$ where $v=g+2f$ is
the total number of vertices. If we now add a vertex and connect it in an arbitrary way to each of the components, so that the graph becomes connected, then the eigenvalues of the new graph and the eigenvalues of the original graph
interlace \cite{HaeInterlacing}, which implies that it has spectrum
$$\{\theta_0,\sqrt{e}^{(f-1)},\theta_f,0^{(g-1)},\theta_{v-f},-\sqrt{e}^{(f-1)},\theta_v\},$$ where we listed the
eigenvalues in non-increasing order. So there are at most seven distinct eigenvalues, but still the number of valencies is at least about $2f$. However, the number of eigenvalues becomes even smaller if we connect the new vertex to all
vertices of one color class of each of the bipartite components. In that case the graph remains bipartite, and
it is easy to see that the rank of the adjacency matrix of the new graph is the same as that of the original graph. Indeed, consider the incidence matrix between the bipartite color classes. Then the extra row that corresponds to the new vertex is the sum of the other distinct rows. This implies
that $0$ has multiplicity $g+1$, hence $\theta_f=\theta_{v-f}=0$, and the new graph only has five eigenvalues.
Note that $\theta_0=-\theta_v \neq \sqrt{e}$ because the graph is connected.

To prove the statement, we may take $t \geq 3$ without loss of generality. Now take, for example, $e=2^{t-1}$, and start from $f=\lfloor \frac{t-1}2 \rfloor +1$ complete bipartite graphs on $2^i+2^{t-1-i}$ vertices ($i=0,1,...,\lfloor \frac{t-1}2 \rfloor$), to obtain a graph with at least $t$ valencies.
\end{proof}

Is this the best we can do with bipartite graphs? Bipartite graphs with four eigenvalues are precisely the
incidence graphs of so-called uniform multiplicative designs, see \cite{vDS}. Examples of such graphs are
known with up to four distinct valencies. The smallest of such examples is on $28$ vertices and is constructed
from the Fano plane. Its spectrum is $\{\sqrt{72}^{(1)},\sqrt{2}^{(13)},-\sqrt{2}^{(13)},-\sqrt{72}^{(1)}\}$ and its
valencies are $3$, $4$, $10$, and $11$, each occurring seven times. This graph is actually part of an infinite family of bipartite graphs with four eigenvalues and four valencies that can be obtained from a construction of non-normal uniform multiplicative designs by Ryser \cite{Ryser}. No examples of bipartite graphs with four
eigenvalues and more than four valencies are currently known however. Note that there is a strong resemblance between graphs with three eigenvalues and bipartite graphs with four eigenvalues, see \cite{Dam3ev} and \cite{vDS}. The following problem resembles De Caen's problem on graphs with three eigenvalues.

\begin{problem} Are there connected bipartite graphs with four distinct eigenvalues and more than four distinct valencies ?
\end{problem}

Note that the above graphs with five eigenvalues can be interpreted as so-called coclique extensions of the spider
graph. The spider graph with $f$ legs is a disjoint union of $f$ $K_2$s plus an extra vertex that is adjacent to one vertex of each
$K_2$; it has spectrum $\{\sqrt{f+1}^{(1)},1^{(f-1)},0^{(1)},-1^{(f-1)},-\sqrt{f+1}^{(1)}\}$. Indeed, if we replace each vertex of this graph by several copies of this vertex, and let two copies be adjacent if and only if their originals are adjacent, then we obtain the required graph. In the adjacency matrix, this means that each zero is  replaced by a block of zeros and each one by a block of ones, in such a way that the rank of the matrix does not change. Hence this produces graphs with relatively small rank. It might in fact be fruitful to consider graphs with small rank and many vertices, as studied by Akbari, Cameron, and Khosrovshahi \cite{ACK} and Haemers and Peeters \cite{HP}.

\section{Non-bipartite graphs with five eigenvalues}

In the same spirit as above, we can construct non-bipartite graphs with five eigenvalues and arbitrarily many valencies.

\begin{theorem} For every integer $t \geq 1$, there is a connected non-bipartite graph with five distinct eigenvalues and exactly $t$ distinct valencies.
\end{theorem}

\begin{proof}
Consider again the disjoint union of $f$ mutually non-isomorphic complete bipartite graphs with $e$ edges, with $f>1$. Take its complement $\G$, which is clearly connected and which has $2f-1$ or $2f$ valencies. Let $A$ be the adjacency matrix of $\G$, then it is easy to see that $A+I$ has precisely $2f$ distinct rows, and these are linearly independent. Therefore $\G$ has eigenvalue $-1$ with multiplicity $v-2f$, where $v$ is the number of vertices. It is also not so hard to see that $\G$ has eigenvalues $\pm \sqrt{e} -1$, each with multiplicity at least $f-1$. Indeed, if a graph has eigenvalue $\theta$ with multiplicity $m$, then its complement has eigenvalue $-1-\theta$ with multiplicity at least $m-1$, because the eigenspace of $\theta$ intersects the orthogonal complement of the all-ones vector in a subspace of dimension $m-1$ or $m$, and the nonzero vectors in this intersection are easily seen to be eigenvectors of the complement of the graph. Thus, the spectrum of $\G$ is $\{\rho_1^{(1)},-1+\sqrt{e}^{(f-1)},\rho_2^{(1)},-1^{(v-2f)},-1-\sqrt{e}^{(f-1)}\}$, where $\rho_1$ is the spectral radius, and $\rho_2$ is the remaining eigenvalue. By considering that $\tr A=0$ and $\tr A^2=v(v-1)-2fe$ (twice the number of edges of $\G$), it follows that $\rho_1+\rho_2=v-2$ and $\rho_1^2+\rho_2^2=v^2-2v+2-2e(2f-1)$. Therefore $$\rho_{1,2}=-1+\frac12 v \pm \frac12 \sqrt{v^2-4e(2f-1)},$$ and it can be shown that these are distinct from the other three eigenvalues of $\G$ because $v>2f\sqrt{e}$; we omit the technical details. So $\G$ has a total of five distinct eigenvalues.

To prove the statement for $t \geq 3$ we can, for example, take $e=2^{t-1}$, and take the complement of the disjoint union of $f=\lfloor \frac{t-1}2 \rfloor +1$ complete bipartite graphs on $2^i+2^{t-1-i}$ vertices ($i=0,1,...,\lfloor \frac{t-1}2 \rfloor$), to get a graph with exactly $t$ valencies. To finish the proof, we need non-bipartite graphs with five eigenvalues that have one and two distinct valencies; examples are the Hamming graph $H(4,3)$ and the complement of $2K_{1,2}$, respectively.
\end{proof}

\section{Regular two-graphs}

In this section, we will use so-called regular two-graphs to construct (connected, non-bipartite) graphs with four eigenvalues. We will now recall some basics on two-graphs; for more details we refer to \cite{bh12, GoRobook, Seidel}.

Let $\G$ be a graph with adjacency matrix $A$. Its Seidel matrix $S= S(\G)$ is defined as $J-I-2A$.
Let $\Pi= \{U, W\}$ be a two-partition of the vertex set $V$ of $\G$. We say a graph --- denoted by $\G^{\Pi}$ --- with the same vertex set as $\G$ is obtained by Seidel switching $\G$ with respect to $\Pi$ if two distinct vertices $x$ and $y$ are adjacent in $\G^{\Pi}$
precisely if $x$ and $y$ are adjacent in $\G$ and either both are in $U$ or both are in $W$, or if they are not adjacent in $\G$ and one of them is in $U$ and the other one is in $W$. In other words, the edges and non-edges between $U$ and $W$ have been switched. It is well-known that the spectra of $S(\G)$ and $S(\G^{\Pi}$) are the same.
The switching class $[\G]$ of $\G$ is the set $$\{ \G^{\Pi} \mid \Pi \textrm{ is a two-partition of the vertex set of }\G,\textrm{ possible with one part empty}\}.$$ Note that switching induces an equivalence relation on graphs, with switching classes as equivalence classes. There is a one-to-one correspondence between switching classes of graphs and so-called two-graphs. For the sake of readability however, we will simply call the switching class a two-graph.

We say that the two-graph $[\G]$ is regular if the Seidel matrix $S(\G)$ has exactly two eigenvalues. Note that if the number of vertices of $\G$ is at least two, then the Seidel matrix $S(\G)$ has at least two eigenvalues. The regular two-graphs containing a complete graph or an empty graph are called trivial.

The graphs in regular two-graphs are examples of so-called strong graphs. We are going to use these strong graphs to show that there are connected graphs with many distinct valencies and exactly four distinct eigenvalues.

In the following, we consider a graph $\G$ in a regular two-graph with $v$ vertices. Let the Seidel matrix of $\G$ have distinct eigenvalues $-1-2\sigma$ and $-1-2\tau$, with respective multiplicities $m_{\sigma}$ and $m_{\tau}$.
First we will derive some more basic properties of $\G$.

\begin{lemma}\label{lemedge}
Let $\G$ be a graph in a regular two-graph with $v$ vertices and Seidel eigenvalues $-1-2\sigma$ and $-1-2\tau$, with respective multiplicities $m_{\sigma}$ and $m_{\tau}$. If $\G$ has $e$ edges, then the (adjacency) spectrum of $\G$ is  $\{\rho_1^{(1)},\rho_2^{(1)},\sigma^{(m_{\sigma}-1)},\tau^{(m_{\tau}-1)}\}$, where $\rho_1$ and $\rho_2$ are not necessarily distinct from each other or $\sigma$ or $\tau$, and the following equations hold:

\begin{align}\label{basicequations}
&m_{\sigma}+m_{\tau}=v,\notag \\
&m_{\sigma}\sigma+m_{\tau}\tau=-v/2,\notag \\
&m_{\sigma}\sigma^2+m_{\tau}\tau^2=v^2/4,\\
&\rho_1+\rho_2=\sigma+\tau+v/2=-2\sigma\tau,\notag\\
&\rho_1^2+\rho_2^2=\sigma^2+\tau^2+2e-v^2/4\notag.
\end{align}
\end{lemma}
\begin{proof}
The adjacency matrix $A=\frac12(J-I-S)$ of $\G$ has eigenvalue $\sigma$ with multiplicity at least $m_{\sigma}-1$ and eigenvalue $\tau$ with multiplicity at least $m_{\tau}-1$ because, similar as before, the eigenspaces of $S$ intersect the orthogonal complement of the all-ones vector in spaces of dimension at least $m_{\tau}-1$ and $m_{\sigma}-1$.
So we have two unknown eigenvalues, say $\rho_1$ and $\rho_2$. But the sum of the eigenvalues ($\tr A$) equals zero and the sum of squares of the eigenvalues ($\tr A^2$) equals twice the number of edges. The given equations follow from these sums, and from using that $\tr S =0$ and $\tr S^2=v(v-1)$. We also use the well-known fact that $v-1=-(2\sigma+1)(2\tau+1)$, which follows from the equation
\begin{equation}\label{quadraticeq}(S+(2\sigma+1)I)(S+(2\tau+1)I)=0,
\end{equation}
and the fact that the diagonal entries of $S^2$ are all $v-1$.
\end{proof}

It follows in particular that, within the switching class $[\G]$, the spectrum of the graph is determined by the number of edges $e$. Because $e \leq {v \choose 2}$ and the number of graphs in $[\G]$ is $2^{v-1}$, there are many graphs in the switching class that have the same spectrum. It is unclear, however, how many of these graphs are non-isomorphic. We will address this point in more detail after Theorem \ref{thm:9}.

We remark that for non-trivial regular two-graphs, the eigenvalues $\sigma$ and $\tau$ cannot be $0$ or $-1$, and the multiplicities $m_{\sigma}$ and $m_{\tau}$ are larger than $1$ (cf.~\cite[Thm.~6.6]{Seidel}). We also note that if $\G$ is regular, then all eigenvectors of $S$ are also eigenvectors of $A$, and it follows that $\G$ has at most three distinct eigenvalues, so if the regular two-graph is non-trivial, then $\G$ is strongly regular. In the following, we will show that if $\G$ is non-regular, then $\G$ has four distinct eigenvalues. We will also show that $\G$ cannot be bipartite, but first, we will characterize the case that $\G$ is disconnected.

We note that by switching in any graph $\G$ with respect to $\Pi=(U,V \setminus U)$, where $U$ is the set of neighbors in $\G$ of a given vertex $u$, one can always isolate $u$, in the sense that it has no neighbors, and hence the graph  $\G^{\Pi}$ is disconnected. If the switching class is a regular two-graph, then it is well-known that $\G^{\Pi}$ is the disjoint union of a vertex and a connected strongly regular graph with parameters $(v,k,\lambda,\mu)$ with $k=2\mu$, see \cite{bh12, GoRobook}. In fact, regular two-graphs are characterized by this property. In the following proposition, we will extend this result, in the sense that we will show that there can be no other disconnected graphs in the switching class.

\begin{prop}\label{disco}
Let $\G$ be a disconnected graph in a non-trivial regular two-graph, with Seidel eigenvalues $-1-2\sigma$ and $-1-2\tau$. Then $\G$ is the disjoint union of an isolated vertex and a connected strongly regular graph with parameters $(-(2\sigma+1)(2\tau +1), -2\sigma\tau, \sigma + \tau -\sigma\tau, -\sigma\tau )$.
\end{prop}

\begin{proof} Consider one of the connected components on the set, say, $V_1$ of $v_1$ vertices, and let $V_2$ be the set of remaining $v_2=v-v_1$ vertices. Then the Seidel matrix $S$ partitions accordingly as
$$S=\begin{bmatrix} S_1 & J\\ J^{\top} & S_2 \end{bmatrix},$$ where $J$ is the $v_1 \times v_2$ all-ones matrix. Since the Seidel matrix satisfies \eqref{quadraticeq},
it follows by considering the upper right block that
\begin{equation}\label{linesums}
S_1J+JS_2+(2\sigma+2\tau+2)J=0.
\end{equation}
By considering a column of this equation, it follows that $S_1$ has constant row sum, say $c_1$. Similarly, it follows that $S_2$ has constant column sum, say $c_2$. Clearly, this means that the respective induced graph $\G_i$ is regular with valency $k_i=(v_i-1-c_i)/2$, for $i=1,2$, respectively. By \eqref{basicequations} and \eqref{linesums}, it follows that $k_1+k_2=(v-2-c_1-c_2)/2=v/2+\sigma+\tau=\rho_1+\rho_2$. Note also that $k_1$ and $k_2$ are eigenvalues of $G$, so they are contained in the spectrum $\{\rho_1^{(1)},\rho_2^{(1)},\sigma^{(m_{\sigma}-1)},\tau^{(m_{\tau}-1)}\}$.  Moreover, we may assume without loss of generality that $\rho_1$ is an eigenvalue of $\G_1$ and $\rho_2$ is an eigenvalue of $\G_2$, for otherwise the `component' not containing either of them has at most two distinct eigenvalues ($\sigma$ and $\tau$), but these are not $0$ or $-1$, which is a contradiction. So both $\G_1$ and $\G_2$ are regular graphs with at most three eigenvalues. For the same reason, $\G_2$ must be connected (recall that we already assumed that $\G_1$ is connected), for otherwise it would have a connected component with at most two distinct eigenvalues $\sigma$ and $\tau$. Because $\G_i$ is $k_i$-regular, it follows that it has spectral radius $k_i$, and hence $\rho_i \leq k_i$, for $i=1,2$. From the fact that $k_1+k_2=\rho_1+\rho_2$, it now follows that $k_1=\rho_1$ and $k_2=\rho_2$. Thus, $\G_i$ is a connected $\rho_i$-regular graph with at most three distinct eigenvalues, for $i=1,2$. However, because $-1$ is not an eigenvalue, neither component can be a clique with at least two vertices. Now two cases remain.

First, if one of the two components, say $\G_2$, is an isolated vertex, then $\rho_2=k_2=0$, and $\G_1$ is a connected strongly regular graph with $\rho_1=k_1=-2\sigma\tau$ by \eqref{basicequations}. Now denote the parameters of $\G_1$ by $(v_1,k_1,\lambda_1,\mu_1)$. Then $v_1=v-1=-(2\sigma+1)(2\tau+1)$, $\mu_1=k_1+\sigma\tau=-\sigma\tau$, and $\lambda_1=\sigma+\tau+\mu_1= \sigma+\tau -\sigma\tau$.

Secondly the case remains that both components are connected strongly regular graphs. Assume without loss of generality that $\G_2$ has the smallest valency of the two components, and let it have parameters $(v_2,k_2,\lambda_2,\mu_2)$. Because $k_1+k_2=-2\sigma\tau$ by \eqref{basicequations}, it follows that $k_2\leq -\sigma\tau$. But then $\mu_2=k_2+\sigma\tau \leq 0$, and so $\G_2$ is disconnected, which is a contradiction. Thus, this final case cannot occur.
\end{proof}

We note that an alternative, more combinatorial, proof of this result is possible if one uses the above mentioned correspondence between regular two-graphs and strongly regular graphs. The given proof, however, is self-contained, and moreover establishes this correspondence.

Another consequence of the correspondence to strongly regular graphs is that $\sigma$ and $\tau$ are integers, except (possibly) if $m_{\sigma}=m_{\tau}$, in which case $\sigma$ and $\tau$ are equal to $-\frac12\pm\frac12\sqrt{v-1}$.

\begin{prop}\label{prop3ev}
Let $\G$ be a graph with at most three distinct eigenvalues in a non-trivial regular two-graph, with Seidel eigenvalues $-1-2\sigma$ and $-1-2\tau$.
Then $\G$ is strongly regular with parameters $(-(2\sigma+1)(2\tau +1) +1,-\tau(2\sigma +1), \sigma(1-\tau), -\tau(\sigma + 1))$ or $(-(2\sigma+1)(2\tau +1) +1,-\sigma(2\tau +1), \tau(1-\sigma), -\sigma(\tau + 1))$.
\end{prop}
\begin{proof}
By the previous proposition, $\G$ must be connected, so the spectral radius has multiplicity one. Consider the spectrum $\{\rho_1^{(1)},\rho_2^{(1)},\sigma^{(m_{\sigma}-1)},\tau^{(m_{\tau}-1)}\}$ of $\G$, see Lemma \ref{lemedge}. By the assumption that $\G$ has at most three distinct eigenvalues, we may assume without loss of generality that $\rho_2=\sigma$ or $\rho_2=\rho_1$.

Suppose first that $\rho_2=\sigma$. Then $\rho_1=\tau+v/2$ by \eqref{basicequations}. Because $m_{\sigma}>1$ and $\rho_1>\tau$, it follows that $\rho_1$ must be the spectral radius. Because $2e= \tr A^2 =\rho_1^2+m_{\sigma} \sigma^2 +(m_{\tau}-1)\tau^2=\rho_1^2+v^2/4-(\rho_1-v/2)^2=v\rho_1$, where we used \eqref{basicequations}, it now follows that $\G$ is regular. Thus, $\G$ is strongly regular, and its parameters $(-(2\sigma+1)(2\tau +1) +1,-\sigma(2\tau +1), \tau(1-\sigma), -\sigma(\tau + 1))$ follow in a straightforward manner. By interchanging the role of $\tau$ and $\sigma$ we obtain the other parameter sets in the statement of the proposition.

Finally, suppose that $\rho_2=\rho_1$. Then $\rho_1=\rho_2=-\sigma\tau$ by \eqref{basicequations}. In this case, assume (without loss of generality) that $\tau>\sigma$. Then $\tau$ must be the spectral radius, and $\sigma<-1$. But then $\rho_1=-\sigma\tau>\tau$, which is a contradiction.
\end{proof}

\begin{prop}\label{bipartite}
Let $\G$ be a graph in a non-trivial regular two-graph. Then $\G$ is not bipartite.
\end{prop}

\begin{proof} Suppose that $\G$ is bipartite. By the previous two propositions it follows that $\G$ is connected with four distinct eigenvalues. From the bipartiteness, we have that its spectrum $\{\rho_1^{(1)},\rho_2^{(1)},\sigma^{(m_{\sigma}-1)},\tau^{(m_{\tau}-1)}\}$ is symmetric about $0$.  If $m_{\sigma}>2$ or $m_{\tau}>2$, then it follows that $\sigma=-\tau$ and $\rho_1=-\rho_2$, but then the equation $\rho_1+\rho_2=\sigma+\tau+v/2$ from \eqref{basicequations} gives a contradiction. So $m_{\sigma}= 2$ and $m_{\tau}= 2$, and hence $v=4$. However, the only regular two-graphs on four vertices are the trivial ones.
\end{proof}

Besides being of general interest, Propositions \ref{disco}-\ref{bipartite} allow us to conclude that if we find graphs in regular two-graphs with more than two valencies, then they are connected, non-bipartite, and have four distinct eigenvalues. In the next section we will indeed construct such graphs, with arbitrarily many valencies.

\section{Graphs with four eigenvalues}

Let $r$ be a positive integer.
Let $\langle \cdot, \cdot \rangle$ denote a non-degenerate symplectic bilinear form on $V= $ GF$(2)^{2r}$.
Let $\G$ be the graph with vertex set $V$ and $u \sim v$ if $\langle u,v \rangle \neq 0$. The switching class $[\G]$ is known as the symplectic two-graph; and it is regular with $\sigma,\tau=\pm 2^{r-1}$. It is clear that $\G$ has $0$ as an isolated vertex. The other component of $\G$ is known as the symplectic graph Sp$(2r)$, which is a strongly regular graph with parameters $(2^{2r}-1,2^{2r-1},2^{2r-2},2^{2r-2})$ according to Proposition \ref{disco}, see also \cite[Lemma 10.12.1]{GoRobook}.

We will use the fact that the graph Sp$(2r)$ has every graph on at most $2r-1$ vertices as an induced subgraph, a result shown by Vu \cite{vu}, see also \cite[Thm.~8.11.2]{GoRobook}.
\begin{theorem}\label{thm:symplectic}
Let $t \geq 3$, and $\Delta$ be a graph on $n$ vertices with $t$ distinct valencies. Then there exists a connected graph on at most $2^{n+2}$ vertices with four distinct eigenvalues and at least $t$ distinct valencies, having $\Delta$ as an induced subgraph.
\end{theorem}
\begin{proof}
Let $r = \lfloor \frac{n}{2} \rfloor+1$ and consider the symplectic two-graph $[\G]$ as described above. Then $\Delta$ is an induced subgraph of the component Sp$(2r)$ of $\G$, say on the vertex set $U$. Now switch $\G$ with respect to $\Pi=(U,V \setminus U)$. If a vertex $u\in U$ has valency $d_u$ in $\Delta$, then in $\G$ it has $2^{2r-1}-d_u$ neighbors in $V \setminus U$, and hence it follows that $u$ has valency $2^{2r}-2^{2r-1}+2d_u$ in $\G^{\Pi}$. Thus the resulting graph $\G^{\Pi}$ has at least $t$ valencies, and as the switching class of $\Gamma$ is a regular two-graph, we obtain that $\G^{\Pi}$ is connected with exactly four distinct eigenvalues, by Lemma \ref{lemedge} and Propositions \ref{disco} and \ref{prop3ev}.
\end{proof}

Note that also the Paley graphs of large enough order have the property that they contain all graphs on a given number of vertices as induced subgraphs (see \cite[Thm.~3]{BolTho}), so instead of the symplectic two-graphs, one can also use the regular two-graphs that correspond to the Paley graphs. In any case, we conclude the following.

\begin{theorem}\label{thm:9} For every integer $t$, there exists a connected non-bipartite graph with four distinct eigenvalues and at least $t$ distinct valencies.
\end{theorem}

As a side result of the construction method presented in the proof of Theorem \ref{thm:symplectic}, we obtain that it is possible to construct non-isomorphic graphs with the same spectrum, starting from non-isomorphic subgraphs with the same number of vertices and edges for $\Delta$ (see also the remark after Lemma \ref{lemedge}). This follows in this particular case from the fact that $n$ is small compared to $k=2^{2r-1}$, so that in $\G^{\Pi}$ there is a unique vertex with valency $n$, and its local graph is $\Delta$.

A computational experiment shows that among the $2^{15}$ graphs in the symplectic two-graph on 16 vertices, there are at least four connected non-isomorphic graphs with the same spectrum as $\G$, which itself is not connected. In total 15 possible spectra occur, and for each of these except for two, we obtain at least two non-isomorphic graphs with that particular spectrum. In one exceptional case we obtain the strongly regular Clebsch graph, and in the other we obtain a graph with spectral radius $4+\sqrt{27} \approx 9.1962$ and valencies $7$, $9$, $11$, and $13$. The maximum number of distinct valencies for a graph in this regular two-graph is five. For details, we refer to Table \ref{table:16}, where we list the possible spectral radii (first column) and sequences of valencies (middle column), and how many times each of these occur (last column). Note that the same sequence of valencies can be shared by non-isomorphic graphs. The graphs with spectral radius $6.8284$ all have the same sequence of valencies, but they are not all isomorphic because for some of these, the two vertices with valency four are adjacent, while for others they are not. We managed however to show that the $960$ graphs with spectral radius $9.1962$ are all isomorphic. Note that if $\rho$ is the spectral radius, then the full spectrum is given by $\{\rho^{(1)},2^{(5)},8-\rho^{(1)},-2^{(9)}\}$, see Lemma \ref{lemedge}.

\begin{table}
\begin{center}
\begin{tabular}{|c|c|c|}
\hline
spectral radius & valencies & number \\[1ex]
\hline
$6$ & $6^{(16)}$  &  $70$  \\[1ex]
$6.6458$ & $3^{(1)}, 5^{(3)}, 7^{(12)}$   & $240$ \\[1ex]  
$6.6458$ & $5^{(6)}, 7^{(9)}, 9^{(1)}$   & $1120$ \\[1ex]  
$6.8284$ & $4^{(2)}, 6^{(8)}, 8^{(6)}$   & $2160$ \\[1ex]  
$7.3166$ & $3^{(1)}, 5^{(3)}, 7^{(8)}, 9^{(4)}$   & $2880$ \\[1ex]  
$7.3166$ & $5^{(6)}, 7^{(5)}, 9^{(5)}$   & $1152$ \\[1ex]  
$7.4641$ & $2^{(1)}, 6^{(6)}, 8^{(8)}, 10^{(1)}$   & $720$ \\[1ex]  
$7.4641$ & $4^{(4)}, 8^{(12)}$   & $240$ \\[1ex]  
$7.4641$ & $4^{(2)}, 6^{(6)}, 8^{(6)}, 10^{(2)}$   & $3360$ \\[1ex]  
$7.8730$ & $1^{(1)}, 7^{(9)}, 9^{(6)}$   & $240$ \\[1ex]  
$7.8730$ & $3^{(1)}, 5^{(2)}, 7^{(7)}, 9^{(5)}, 11^{(1)}$   & $2880$ \\[1ex]  
$7.8730$ & $5^{(5)}, 7^{(4)}, 9^{(6)}, 11^{(1)}$   & $1440$ \\[1ex]  
$8$ & $0^{(1)}, 8^{(15)}$   & $16$ \\[1ex]  
$8$ & $2^{(1)}, 6^{(4)}, 8^{(8)}, 10^{(3)}$   & $960$ \\[1ex]  
$8$ & $4^{(3)}, 8^{(12)}, 12^{(1)}$ & $240$ \\[1ex]  
$8$ & $4^{(2)}, 6^{(4)}, 8^{(6)}, 10^{(4)}$   & $2880$ \\[1ex]  
$8$ & $6^{(10)}, 10^{(6)}$   & $192$ \\[1ex]  
$8.3589$ & $3^{(1)}, 5^{(2)}, 7^{(3)}, 9^{(9)}, 11^{(1)}$   & $960$ \\[1ex]  
$8.3589$ & $3^{(1)}, 7^{(9)}, 9^{(3)}, 11^{(3)}$   & $320$ \\[1ex]  
$8.3589$ & $5^{(3)}, 7^{(6)}, 9^{(4)}, 11^{(3)}$   & $1920$ \\[1ex]  
$8.4721$ & $4^{(1)}, 6^{(4)}, 8^{(6)}, 10^{(4)}, 12^{(1)}$   & $2880$ \\[1ex]  
$8.4721$ & $6^{(8)}, 10^{(8)}$   & $180$ \\[1ex]  
$8.7958$ & $5^{(3)}, 7^{(6)}, 9^{(4)}, 11^{(3)}$   & $1920$ \\[1ex]  
$8.7958$ & $5^{(2)}, 7^{(4)}, 9^{(8)}, 11^{(1)}, 13^{(1)}$   & $720$ \\[1ex]  
$8.8990$ & $4^{(2)}, 8^{(6)}, 10^{(8)}$   & $240$ \\[1ex]  
$8.8990$ & $6^{(4)}, 8^{(6)}, 10^{(4)}, 12^{(2)}$   & $1440$ \\[1ex]  
$9.1962$ & $7^{(6)}, 9^{(6)}, 11^{(3)}, 13^{(1)}$   & $960$ \\[1ex]  
$9.2915$ & $6^{(2)}, 8^{(6)}, 10^{(6)}, 12^{(2)}$   & $480$ \\[1ex]  
$9.2915$ & $6^{(1)}, 8^{(8)}, 10^{(6)}, 14^{(1)}$   & $240$ \\[1ex]  
$9.2915$ & $8^{(12)}, 12^{(4)}$   & $80$ \\[1ex]  
$9.5678$ & $5^{(1)}, 9^{(10)}, 11^{(5)}$   & $96$ \\[1ex]  
$9.5678$ & $9^{(15)}, 15^{(1)}$   & $16$ \\[1ex]  
$10$ & $10^{(16)}$   & $6$ \\[1ex]  \hline
\end{tabular}
\caption{Graphs in the switching class of the symplectic two-graph on 16 vertices.}\label{table:16}
\end{center}
\end{table}

In this context, it is good to mention that Cioab\u{a}, Haemers, Vermette, and Wong \cite{friendship} recently showed that all Friendship graphs except the one on $33$ vertices are determined by the spectrum. These graphs also have four distinct eigenvalues of which two are simple, just like the non-regular graphs in a regular two-graph. Also, Van Dam \cite[Thm.~4.4, \S 4.5.3]{Dam4} characterized several regular graphs with four eigenvalues of which two are simple. Moreover, for most of the latter characterizations, Seidel switching plays a key role. In general, however, we expect that almost all graphs with few (say at most seven) eigenvalues are not determined by the spectrum. We expect this even more for graphs with four eigenvalues of which two are simple. For the general question of which graphs are determined by the spectrum, we refer to \cite{DH03,DH09}.

\section{Graphs with three eigenvalues}

Let us return in this final section to De Caen's original question. Currently, we know only of finitely many connected graphs with three eigenvalues and three valencies. More specifically, Spence (see \cite{Dam3ev}) constructed an example on 24 vertices, De Caen, Van Dam, and Spence \cite{dCvDS} constructed examples on 36 and 43 vertices, and Bridges and Mena \cite{BM} constructed examples on 46 and 97 vertices. Let us see whether we can find an example with three eigenvalues and four valencies though.

Consider a connected graph with three distinct eigenvalues. As long as the number of distinct valencies is at most
three, the partition of the vertices according to their valencies is equitable (see \cite{Dam3ev}), and this makes a search for putative parameter sets for such graphs easier. It is unclear whether the `valency partition' is equitable if
the number of valencies is larger than three, or whether such graphs can exist at all. A putative parameter set with four valencies (in fact, the one with the smallest number of vertices according to 15-year old, but unverified, computations) is the following one on $51$ vertices and spectrum $\{30^{(1)},3^{(20)},-3^{(30)}\}$. The computations show that a graph with this spectrum must have valencies $13$, $18$, $34$, and $45$, occurring $15$, $5$, $30$, and $1$ times, respectively. In fact, using the techniques of \cite{Dam3ev} it can be shown that in this particular case, the valency partition is also equitable, with quotient matrix
$$\begin{bmatrix}2 &0& 10& 1\\0 &0& 18&0 \\5&3&25&1 \\15&0&30&0\end{bmatrix}.$$
Quite a bit of this graph is therefore determined. Besides the trivial parts, one can show that the incidence structure between the five vertices of valency $18$ and $30$ vertices of valency $34$ is a $2$-$(5,3,9)$ design, and there is only one such design: three times the full design of all triples on five points. We leave it as a problem to the reader to finish the (de-)construction.\\

\bigskip
{\bf Acknowledgment.} The authors thank Sebastian Cioab\u{a}, Willem Haemers, and  Bojan Mohar for a discussion on the topic of this paper during a meeting in Durham, July 2013. They also thank the referees for comments on an earlier version that helped improve the presentation of the results. JHK thanks the Chinese Academy of Sciences for its support under the `100 talents' program.



\begin{thebibliography}{99}

\bibitem{ACK} S. Akbari, P.J. Cameron, and G.B. Khosrovshahi, Ranks and signatures of adjacency matrices
    (preprint 2004); available online at \url{http://www.maths.qmw.ac.uk/~pjc/preprints/ranksign.pdf}.

\bibitem{BolTho} B. Bollob\'{a}s and A. Thomason, Graphs which contain all small graphs, {\sl European J. Combin.} 2 (1981), 13--15.

\bibitem{BM} W.G. Bridges and R.A. Mena,
Multiplicative cones --- a family of three eigenvalue graphs,
{\sl Aequationes Math.} 22 (1981), 208--214.

\bibitem{bh12} A.E. Brouwer and W.H. Haemers, {\sl Spectra of Graphs},
    Springer,
    2012; available online at
    \url{http://homepages.cwi.nl/~aeb/math/ipm/}.

\bibitem{dCvDS} D. de Caen, E.R. van Dam, and E. Spence, A nonregular analogue of conference graphs,
    {\sl J. Combin. Theory Ser. A} 88 (1999), 194--204.

\bibitem{friendship} S.M. Cioab\u{a}, W.H. Haemers, J. Vermette, and W. Wong, The graphs with all but two eigenvalues equal to $\pm 1$, {\sl J. Algebraic Combin.} (to appear);
    arXiv:\href{http://arxiv.org/abs/1310.6529}{1310.6529}.

\bibitem{cds82} D.M. Cvetkovi\'c, M. Doob, and H. Sachs, {\sl Spectra of
    Graphs}, VEB Deutscher Verlag der Wissenschaften, Berlin,
    second edition, 1982.

\bibitem{Dam4} E.R. van Dam, Regular graphs with four eigenvalues, {\sl Linear Algebra
    Appl.} 226-228 (1995), 139--162.

\bibitem{Dam3ev} E.R. van Dam, Nonregular graphs with three eigenvalues, {\sl J. Combin. Theory Ser. B} 73
    (1998), 101--118.

\bibitem{Dom} E.R. van Dam, The combinatorics of Dom de Caen, {\sl Des. Codes Cryptogr.} 34 (2005),
    137--148.

\bibitem{mumubar} E.R. van Dam and W.H. Haemers, Graphs with constant $\mu$ and $\overline{\mu}$, {\sl Discrete Math.} 182 (1998), 293--307.

\bibitem{DH03} E.R. van Dam and W.H. Haemers, Which graphs are
    determined by their spectrum?, {\sl Linear Algebra Appl.}
    373 (2003), 241--272.

\bibitem{DH09} E.R. van Dam and W.H. Haemers, Developments on
    spectral characterizations of graphs, {\sl Discrete Math.}
    309 (2009), 576--586.

\bibitem{vDS} E.R. van Dam and E. Spence, Combinatorial designs with two singular values I. Uniform
    multiplicative designs, {\sl J. Combin. Theory Ser. A} 107 (2004), 127--142.

\bibitem{GoRobook} C. Godsil and G. Royle, {\sl Algebraic Graph Theory}, Springer, 2001.

\bibitem{HaeInterlacing} W.H. Haemers, Interlacing eigenvalues and graphs,
    {\sl Linear Algebra
    Appl.} 226-228 (1995), 593--616.

\bibitem{HP} W.H. Haemers and M.J.P. Peeters, The maximum order of adjacency matrices of graphs with a given
    rank, {\sl Des. Codes Cryptogr.} 65 (2012), 223--232.

\bibitem{Ryser}
H.J. Ryser,
Symmetric designs and related configurations,
{\sl J. Combin. Theory Ser. A} 12 (1972), 98-111.

\bibitem{Seidel} J.J. Seidel, A survey of two-graphs, {\sl Colloquio Internazionale sulle Teorie Combinatorie} (Proceedings, Rome, 1973), Vol.~I, pp.~481-511. Accademia Nazionale dei Lincei, Rome, 1976.

\bibitem{vu} V.H. Vu, A strongly regular $N$-full graph of small order, {\sl Combinatorica} 16 (1996), 295--299.


\end{thebibliography}
\end{document}